\documentclass[a4paper,11pt]{amsart}

\usepackage{amsmath,amssymb,amsthm}
\usepackage{cite}
\allowdisplaybreaks[2]
\numberwithin{equation}{section}
\newtheorem{theorem}{Theorem}[section]
\newtheorem{proposition}[theorem]{Proposition}
\newtheorem{lemma}[theorem]{Lemma}
\newtheorem{corollary}[theorem]{Corollary}
\theoremstyle{definition}
\newtheorem{definition}[theorem]{Definition}
\newtheorem{remark}[theorem]{Remark}
\theoremstyle{plain}
\newcommand{\CP}{\mathbb{C}P}
\newcommand{\CH}{\mathbb{C}H}
\newcommand{\Cn}{\mathbb{C}^{n}}
\newcommand{\grad}{\operatorname{grad}}
\newcommand{\tr}{\operatorname{trace}}
\title[Biharmonic Lagrangian $H$-umbilical submanifolds]
{Biharmonic Lagrangian $H$-umbilical submanifolds
in complex space forms}

\author[Toru Sasahara]{Toru Sasahara}

\date{}

\begin{document}

\begin{abstract}
We give a complete classification of proper biharmonic Lagrangian
$H$-umbilical submanifolds in complex space forms.
Such submanifolds exist only in complex
projective space and are locally congruent to an explicit model
obtained via the Hopf fibration. In particular, their nonexistence
in complex Euclidean space provides a partial
affirmative answer to Chen's conjecture.
\end{abstract}


\def\subjclassname{\textup{2020} Mathematics Subject Classification}

\makeatletter
\let\originalsettitle\@settitle
\def\@settitle{\setlength{\topskip}{12pt}\originalsettitle}
\makeatother
\maketitle
\vspace{-10pt}

\section{Introduction}

An immersed submanifold is called biharmonic if its defining isometric immersion is a critical point of the bienergy functional~\cite{Jiang1986}. A biharmonic submanifold that is not minimal is called proper biharmonic.
Biharmonic submanifolds in real space forms behave quite differently according as the ambient curvature is nonpositive or positive. Chen conjectured that every biharmonic submanifold of Euclidean space must be minimal. The conjecture remains open, but several partial results are known. For example, Fu, Hong, and Zhan proved that every biharmonic hypersurface in $\mathbb{R}^6$ is minimal~\cite{FuHongZhan2023}.

In the unit sphere, the small hypersphere $S^{m}(1/\sqrt{2})$ and the generalized Clifford tori $S^{p}(1/\sqrt{2}) \times S^{q}(1/\sqrt{2})$, with $p+q=m$ and $p\neq q$, are the basic examples of proper biharmonic hypersurfaces. Every proper biharmonic hypersurface with at most two distinct principal curvatures is locally congruent to one of these models~\cite{BalmusMontaldoOniciuc2008}. More recently, Andronic, Fu, and Oniciuc completed the classification for hypersurfaces with at most three distinct principal curvatures: if $m\ge4$, every such proper biharmonic hypersurface in $S^{m+1}$ is locally congruent to one of the same two standard models~\cite{AndronicFuOniciuc2025}.

Biharmonic submanifolds in complex space forms behave differently from those in real space forms. Fetcu, Loubeau, Montaldo, and Oniciuc related the bitension field of a submanifold of $\CP^n$ to that of its Hopf tube. They constructed new families of proper biharmonic submanifolds. They also characterized proper biharmonic curves in terms of their curvatures and complex torsions~\cite{FetcuLoubeauMontaldoOniciuc2010}. For real hypersurfaces in $\CP^n$, proper biharmonic Hopf hypersurfaces in $\CP^2$ and proper biharmonic hypersurfaces with two distinct principal curvatures have been classified. It has also been shown that every biharmonic ruled real hypersurface must be minimal~\cite{Sasahara2019}. The same paper also gives the corrected classification of proper biharmonic homogeneous real hypersurfaces in $\CP^n$. Inoguchi and the author subsequently revisited this classification and classified biharmonic homogeneous real hypersurfaces in quaternionic projective space~\cite{InoguchiSasahara2025}.

Every totally umbilical Lagrangian submanifold of dimension at least
two in a complex space form is totally geodesic~\cite{Chen2001}.
Chen introduced Lagrangian $H$-umbilical submanifolds as a natural
generalization~\cite{Chen1997b}. Biminimal Lagrangian $H$-umbilical
submanifolds with nonzero constant mean curvature were
classified in~\cite{BiminimalHumbilical2012}. Maeta and Urakawa
derived the biharmonic equations for Lagrangian submanifolds and
obtained a classification in the $H$-umbilical case, assuming that
the normalized mean curvature vector is
parallel~\cite{MaetaUrakawa2013}.
In this paper, we show that biharmonicity forces the length of the mean curvature vector to be constant. Using this fact together with known classification results, we obtain a complete classification of proper biharmonic Lagrangian $H$-umbilical submanifolds in complex space forms.

Let $\widetilde M^{n}(4\epsilon)$ denote a complex space form of complex dimension $n$ and constant holomorphic sectional curvature $4\epsilon$. After a homothety, we may normalize the nonflat cases with $\epsilon=1$ or $\epsilon=-1$. Let $\pi : S^{2n+1}(1) \to \CP^{n}(4)$ be the Hopf fibration. Our main result is the following classification theorem.
\begin{theorem}\label{thm:classification-intro}
Let
$f\colon M^n\to\widetilde M^n(4\epsilon)$ be a Lagrangian
$H$-umbilical immersion, where $n\ge2$ and
$\epsilon\in\{-1,0,1\}$. If $f$ is proper biharmonic, then
$\epsilon=1$ and $f$ is locally congruent to an open part of the
immersion
$F_\nu\colon\mathbb R\times S^{n-1}\to\CP^n(4)$ given by
 \begin{equation}\label{eq:model-intro}
 F_\nu(x,y_1,\ldots,y_n)=\pi\Biggl(
 \sqrt{\frac{\nu^2}{\nu^2+1}}e^{-ix/\nu},
 \sqrt{\frac{1}{\nu^2+1}}e^{i\nu x}y_1,\ldots,
 \sqrt{\frac{1}{\nu^2+1}}e^{i\nu x}y_n\Biggr),
 \end{equation}
where $y_1^2+\cdots+y_n^2=1$ and
 \begin{equation}\label{eq:nu-general}
 \nu^2=\frac{n+5\pm\sqrt{n^2+6n+25}}{2n}.
 \end{equation}
Conversely, every immersion in this family is proper biharmonic.
\end{theorem}

The immersion in~\eqref{eq:model-intro} is obtained from the Legendre curve $\gamma_\nu=(\gamma_1,\gamma_2)$ in
$S^3(1)\subset\mathbb C^2$ defined by
\[
 \gamma_1(x)=\sqrt{\frac{\nu^2}{\nu^2+1}}e^{-ix/\nu},
 \qquad
 \gamma_2(x)=\sqrt{\frac{1}{\nu^2+1}}e^{i\nu x}.
\]
This is a unit-speed Legendre curve, since
$|\gamma_\nu'|=1$ and
$\langle\gamma_\nu',i\gamma_\nu\rangle=0$. The curvature vector of
$\gamma_\nu$ in
$S^3(1)$ is
\[
 \nabla^{S^3}_{\partial_x}\gamma_\nu'
 =\left(\nu-\frac1\nu\right)i\gamma_\nu'.
\]
Thus its signed curvature is $\nu-\nu^{-1}$, and its curvature is
$|\nu-\nu^{-1}|$. The map
\[
 \Psi_\nu(x,y)=\bigl(\gamma_1(x),\gamma_2(x)y\bigr),
 \qquad y\in S^{n-1},
\]
is horizontal, and~\eqref{eq:model-intro} is precisely
$F_\nu=\pi\circ\Psi_\nu$. Thus $F_\nu$ is a Lagrangian
$H$-umbilical immersion obtained from a Legendre curve via the Hopf
fibration~\cite{Chen1997b}.

\section{Preliminaries}\label{sec:preliminaries}

Throughout this paper, all manifolds are assumed to be connected.

Let $M^n$ be an $n$-dimensional submanifold of a Riemannian manifold
$\widetilde M$. We denote by $\nabla$ and $\widetilde\nabla$ the
Levi--Civita connections on $M^n$ and $\widetilde M$, respectively.
The Gauss and Weingarten formulas are given by
\begin{align*}
 \widetilde\nabla_XY&=\nabla_XY+h(X,Y),\\
 \widetilde\nabla_X\xi&=-A_\xi X+\nabla_X^\perp\xi,
\end{align*}
respectively, where $X,Y$ are tangent vector fields, $\xi$ is a normal
vector field, and $h$, $A$, and $\nabla^\perp$ are the second
fundamental form, the shape operator, and the normal connection.
The mean curvature vector field $H$ is defined by
$H=n^{-1}\tr h$. The function $|H|$ is called the mean curvature. If
it vanishes identically, then $M^n$ is called a minimal submanifold.
In particular, if $h$ vanishes identically, then $M^n$ is called a
totally geodesic submanifold.

Let $\widetilde M^n(4\epsilon)$ be a complex space form of complex
dimension $n$ and constant holomorphic sectional curvature
$4\epsilon$. Its curvature tensor $\widetilde R$ is given by
\begin{align*}
 \widetilde R(X,Y)Z=\epsilon\{&
 \langle Y,Z\rangle X-\langle X,Z\rangle Y
 +\langle JY,Z\rangle JX-\langle JX,Z\rangle JY\\
 &+2\langle X,JY\rangle JZ\}.
\end{align*}
Here $\langle\ ,\ \rangle$ is the inner product and $J$ is the complex
structure of $\widetilde M^n(4\epsilon)$. A complete, simply connected
complex space form is holomorphically isometric to $\Cn$,
$\CP^n(4\epsilon)$, or $\CH^n(4\epsilon)$ according as
$\epsilon=0$, $\epsilon>0$, or $\epsilon<0$.

A submanifold of $\widetilde M^n(4\epsilon)$ is called Lagrangian if
$J$ interchanges its tangent and normal spaces. Let $M^n$ be a
Lagrangian submanifold of $\widetilde M^n(4\epsilon)$.
Denote by $R$ the curvature tensor of $M^n$. The equations of Gauss
and Codazzi are given by
\begin{align*}
 \langle R(X,Y)Z,W\rangle
={}&\epsilon\{\langle X,W\rangle\langle Y,Z\rangle
-\langle X,Z\rangle\langle Y,W\rangle\}\\
&+\langle[A_{JZ},A_{JW}]X,Y\rangle,\\
 (\overline\nabla_Xh)(Y,Z)&=(\overline\nabla_Yh)(X,Z),
\end{align*}
respectively, where
\begin{equation*}
 (\overline\nabla_Xh)(Y,Z)
 =\nabla_X^\perp h(Y,Z)-h(\nabla_XY,Z)-h(Y,\nabla_XZ).
\end{equation*}

Let $f\colon(M^n,g)\to(N,\overline g)$ be a smooth map between two
Riemannian manifolds. The tension field $\tau(f)$ of $f$ and the
bienergy $E_2(f;\Omega)$ over a relatively compact domain
$\Omega\subset M^n$ are defined by
\begin{equation*}
 \tau(f)=\tr_g(\nabla df),\qquad
 E_2(f;\Omega)=\frac12\int_\Omega|\tau(f)|^2\,dv_g.
\end{equation*}
If $f$ is a critical point of $E_2(f;\Omega)$ with respect to all
compactly supported variations for every relatively compact domain
$\Omega$, then $f$ is called a biharmonic map. Jiang~\cite{Jiang1986}
proved that $f$ is biharmonic if and only if its bitension field
vanishes identically:
\begin{equation*}
 \tau_2(f)=\sum_{i=1}^n\left\{
 \nabla^f_{e_i}\nabla^f_{e_i}\tau(f)
 -\nabla^f_{\nabla_{e_i}e_i}\tau(f)
 +\overline R\bigl(\tau(f),df(e_i)\bigr)df(e_i)
 \right\}=0.
\end{equation*}
Here $\overline R$ is the curvature tensor of $N$, $\nabla^f$ is the
induced connection on $f^*TN$, and $\{e_i\}$ is a local orthonormal
frame field on $M^n$.
If $f$ is an isometric immersion, then $\tau(f)=nH$. An immersed
submanifold is called biharmonic if its defining isometric immersion
$f$ is a biharmonic map, or equivalently, if $\tau_2(f)=0$. A
nonminimal biharmonic submanifold is called proper biharmonic.
Clearly, every minimal submanifold is biharmonic.
Following~\cite{BiminimalSurface2009,BiminimalHumbilical2012}, an isometric
immersion is called biminimal if the normal component of its
bitension field vanishes. This condition is called \emph{free
biminimality} in~\cite{Loubeau2008}.

\begin{definition}
A non-totally geodesic Lagrangian submanifold $M^n$ in
a complex space form $\widetilde M^n(4\epsilon)$ is called
Lagrangian $H$-umbilical if every point has a neighborhood $W$ on
which there exists an orthonormal frame field
$\{e_1,\ldots,e_n\}$ such that the second fundamental form takes the
following form:
\begin{equation}\label{eq:H-umbilical}
 \begin{aligned}
 h(e_1,e_1)&=\lambda Je_1,\qquad
 h(e_2,e_2)=\cdots=h(e_n,e_n)=\mu Je_1,\\
 h(e_1,e_j)&=\mu Je_j,\qquad
 h(e_j,e_k)=0,\qquad 2\le j\ne k\le n,
 \end{aligned}
\end{equation}
where $\lambda$ and $\mu$ are functions on $W$.
\end{definition}

For a Lagrangian $H$-umbilical submanifold,
\begin{equation*}
 H=aJe_1,\qquad
 a=\frac{\lambda+(n-1)\mu}{n}.
\end{equation*}
For $j>1$, the Gauss equation and~\eqref{eq:H-umbilical} give
\begin{equation}\label{eq:Gauss-H-umbilical}
 \langle R(e_1,e_j)e_j,e_1\rangle
 =\epsilon+\lambda\mu-\mu^2.
\end{equation}

\begin{proposition}\label{prop:biharmonic-equations}
Let $f\colon M^n\to\widetilde M^n(4\epsilon)$ be a Lagrangian isometric
immersion, and let $\{e_1,\ldots,e_n\}$ be a local
orthonormal frame field on $M^n$. Then $f$ is biharmonic if and only if the following equations hold\textup{:}
\begin{gather}
 -\sum_{i=1}^n\left(
 \nabla_{e_i}^\perp\nabla_{e_i}^\perp H
 -\nabla_{\nabla_{e_i}e_i}^\perp H\right)
 +\sum_{i=1}^nh(e_i,A_He_i)
 -(n+3)\epsilon H=0,\label{eq:normal-biharmonic}\\
 2\sum_{i=1}^nA_{\nabla_{e_i}^\perp H}e_i
 +\frac n2\grad |H|^2=0.
 \label{eq:tangential-biharmonic}
\end{gather}
\end{proposition}

\begin{proof}
The normal
and tangential equations follow from~\cite{BiminimalHumbilical2012}
and~\cite{TangentialHumbilical2015}, respectively.
\end{proof}

We put
$\omega_i^j(X)=\langle\nabla_Xe_i,e_j\rangle$.

\begin{lemma}\label{lem:Codazzi-relations}
Let $M^n$ be a Lagrangian $H$-umbilical submanifold in a complex
space form whose second fundamental form is given
by~\eqref{eq:H-umbilical}. Then the following relations hold for
$j>1$\textup{:}
\begin{align}
 &e_j\lambda=(2\mu-\lambda)\omega_j^1(e_1),
 \label{eq:Codazzi-lambda-j}\\
 &e_1\mu=(\lambda-2\mu)\omega_1^2(e_2)
 =\cdots=(\lambda-2\mu)\omega_1^n(e_n),
 \label{eq:Codazzi-mu-one}\\
 &e_j\mu=3\mu\omega_1^j(e_1).
 \label{eq:Codazzi-mu-j}
\end{align}
If $n\ge3$, the following additional relations hold\textup{:}
\begin{align}
 &\mu\omega_1^j(e_1)=0,
 \label{eq:Codazzi-p}\\
 &\mu\omega_1^2(e_2)=\cdots=\mu\omega_1^n(e_n),
 \label{eq:Codazzi-diagonal-mu}\\
 &\mu\omega_1^i(e_j)=0,\qquad i,j>1,\quad i\ne j,
 \label{eq:Codazzi-offdiag-mu}\\
 &(\lambda-2\mu)\omega_1^i(e_j)=0,
 \qquad i,j>1,\quad i\ne j.
 \label{eq:Codazzi-offdiag-lambda}
\end{align}
\end{lemma}

\begin{proof}
These relations follow from the Codazzi equation; see
\cite[(5.2)--(5.8)]{TangentialHumbilical2015}.
\end{proof}

Equations~\eqref{eq:normal-biharmonic}
and~\eqref{eq:tangential-biharmonic} have the following components.

\begin{lemma}\label{lem:biharmonic-components}
Let $f\colon M^n\to\widetilde M^n(4\epsilon)$ be a Lagrangian
$H$-umbilical immersion whose second fundamental form is given
by~\eqref{eq:H-umbilical}. Then $f$ is biharmonic if and only if the
following equations hold\textup{:}
\begin{align}
&\begin{aligned}
&-\sum_{i=1}^ne_i e_i a
 +a\sum_{i,j=1}^n(\omega_1^j(e_i))^2\\
&\quad+\sum_{i,j=1}^n(e_j a)\omega_i^j(e_i)
 +a\{\lambda^2+(n-1)\mu^2-(n+3)\epsilon\}=0,
\end{aligned}
\label{eq:normal-component-one}\\
&\begin{aligned}
&-2\sum_{i=1}^n(e_i a)\omega_1^k(e_i)
 -a\sum_{i=1}^ne_i(\omega_1^k(e_i))\\
&\quad-a\sum_{i,j=1}^n\omega_1^j(e_i)\omega_j^k(e_i)
 +a\sum_{i,j=1}^n\omega_i^j(e_i)\omega_1^k(e_j)=0,
 \qquad k=2,\ldots,n,
\end{aligned}
\label{eq:normal-component-k}\displaybreak[4]\\
&\{3\lambda+(n-1)\mu\}e_1 a
 +2a\mu\sum_{i=2}^n\omega_1^i(e_i)=0,
\label{eq:tangent-component-one}\\
&\{\lambda+(n+1)\mu\}e_j a
 +2a\mu\omega_1^j(e_1)=0,
 \qquad j=2,\ldots,n.
\label{eq:tangent-component-j}
\end{align}
Equations~\eqref{eq:normal-component-one}
and~\eqref{eq:normal-component-k} are the $Je_1$- and
$Je_k$-components of the normal equation, respectively,
while~\eqref{eq:tangent-component-one}
and~\eqref{eq:tangent-component-j} are the $e_1$- and
$e_j$-components of the tangential equation.
\end{lemma}

\begin{proof}
The first two equations are established
in~\cite[Lemma~8]{BiminimalHumbilical2012}, and the last two
follow from~\cite[(5.13)--(5.14)]{TangentialHumbilical2015}.
\end{proof}

\section{Constancy of the mean curvature}\label{sec:cmc}

\begin{proposition}\label{prop:cmc}
Let $n\ge2$ and let
$f\colon M^n\to\widetilde M^n(4\epsilon)$ be a biharmonic Lagrangian
$H$-umbilical immersion. Then $|H|$ is constant on $M^n$.
\end{proposition}

\begin{proof}

Suppose that $\grad |H|\ne0$ at some point of $\{H\ne0\}$. Choose a
neighborhood $U$ of this point and an adapted frame field such that
$e_1=-JH/|H|$. Then $a=|H|>0$ and $\grad a\ne0$ on $U$.
If $\mu$ vanishes on a nonempty open subset of $U$, then
$\lambda=na$, and~\eqref{eq:tangent-component-one}
and~\eqref{eq:tangent-component-j} force
$e_i a=0$ for every $i$, contrary to $\grad a\ne0$. Hence, after
restricting $U$, we may assume that $\mu\ne0$.

We define
\begin{equation*}
 p_j=\langle\nabla_{e_1}e_1,e_j\rangle,\qquad
 k_j=\langle\nabla_{e_j}e_1,e_j\rangle,\qquad j>1.
\end{equation*}
We claim that
\begin{equation}\label{eq:higher-transverse}
 p_j=0,\qquad e_j\lambda=e_j\mu=e_j a=0,\qquad
 k_2=\cdots=k_n=:k.
\end{equation}
For $n\ge3$, equations~\eqref{eq:Codazzi-lambda-j},
\eqref{eq:Codazzi-mu-j},~\eqref{eq:Codazzi-p}, and
\eqref{eq:Codazzi-diagonal-mu}, together with $\mu\ne0$, establish
\eqref{eq:higher-transverse}. Suppose that $n=2$.
Write $p=p_2$ and $k=k_2$. Equations~\eqref{eq:Codazzi-lambda-j}
and~\eqref{eq:Codazzi-mu-j} give
\begin{equation*}
 e_2\lambda=(\lambda-2\mu)p,\qquad
 e_2\mu=3\mu p,
\end{equation*}
and therefore $e_2a=ap$. Equation~\eqref{eq:tangent-component-j},
with $n=j=2$, then yields
\begin{equation*}
 (a+2\mu)e_2a=0.
\end{equation*}
If $e_2a\ne0$ on a nonempty open set, then
$\mu=-a/2$ and $\lambda=5a/2$ there.
Equation~\eqref{eq:Codazzi-lambda-j}, together with $p=e_2a/a$, implies
\begin{equation*}
 \frac52e_2a=\frac72e_2a,
\end{equation*}
a contradiction. Hence $e_2a=0$, and consequently $p=0$.
It follows from \eqref{eq:Codazzi-lambda-j} and~\eqref{eq:Codazzi-mu-j}
that $e_2\lambda=e_2\mu=0$. Thus
\eqref{eq:higher-transverse} holds for every $n\ge2$.

Henceforth, a prime denotes differentiation in the
$e_1$-direction. Equations~\eqref{eq:tangent-component-one}
and~\eqref{eq:Codazzi-mu-one} reduce to
\begin{gather}
 (2\lambda+na)a'+2(n-1)a\mu k=0,
 \label{eq:higher-tangential}\\
 \mu'=(\lambda-2\mu)k.
 \label{eq:higher-muprime}
\end{gather}

Neither $\lambda-2\mu$ nor $2\lambda+na$ vanishes identically on a
nonempty open subset of $U$. Indeed, if $\lambda-2\mu=0$, then
from~\eqref{eq:higher-muprime}, we find $\mu'=0$. Since
$a=(n+1)\mu/n$, necessarily $a'=0$. If
$2\lambda+na=0$, then from~\eqref{eq:higher-tangential}, we have $k=0$, and
equation~\eqref{eq:higher-muprime} yields $\mu'=0$. Since
$a=2(n-1)\mu/(3n)$, we again obtain $a'=0$. Thus the sets on which
these two functions are nonzero are open and dense in $U$. After
restricting $U$, we may assume that
\begin{equation}\label{eq:higher-nonzero-factors}
 \mu(\lambda-2\mu)(2\lambda+na)\ne0.
\end{equation}

Under~\eqref{eq:higher-nonzero-factors}, we introduce
\begin{equation*}
 t=\frac{\mu}{a},\qquad w=\frac{a'}a.
\end{equation*}
Since $\lambda=na-(n-1)\mu=a\{n-(n-1)t\}$,~\eqref{eq:higher-tangential} becomes
\begin{align}
 k&=-\frac{3n-2(n-1)t}{2(n-1)t}w.
 \label{eq:higher-k}
\end{align}
It follows from $\mu=at$ and $\lambda=na-(n-1)\mu$ that
\[
 \mu'=a(t'+tw),\qquad
 \lambda-2\mu=a\{n-(n+1)t\}.
\]
Substituting these identities and~\eqref{eq:higher-k}
into~\eqref{eq:higher-muprime}, we obtain
\begin{align}
 t'&=-\frac{(nt-n-t)(2nt-3n+4t)}
 {2(n-1)t}w.
 \label{eq:higher-tprime}
\end{align}
For $n=2$, the identity $\nabla_{e_2}e_1=ke_2$ is automatic. For
$n\ge3$, since $\mu\ne0$, equation~\eqref{eq:Codazzi-offdiag-mu}
implies $\nabla_{e_j}e_1=ke_j$. Under~\eqref{eq:higher-nonzero-factors},
the same conclusion also follows from
equation~\eqref{eq:Codazzi-offdiag-lambda}. Moreover, from
\eqref{eq:higher-transverse}, we have
$\nabla_{e_1}e_1=0$. By a direct calculation, we find
\begin{equation}\label{eq:higher-intrinsic-curvature}
 \langle R(e_1,e_j)e_j,e_1\rangle=-e_1k-k^2=-k'-k^2.
\end{equation}
Combining~\eqref{eq:higher-intrinsic-curvature}
with~\eqref{eq:Gauss-H-umbilical} yields
\begin{align}
 -k'-k^2&=\epsilon+\mu(\lambda-\mu).
 \label{eq:higher-Gauss}
\end{align}
By~\eqref{eq:higher-transverse}, equation~\eqref{eq:normal-component-one} becomes
\begin{equation}\label{eq:higher-normal}
\begin{aligned}
 &-a''-(n-1)ka'+(n-1)ak^2\\
 &\quad+a\{\lambda^2+(n-1)\mu^2-(n+3)\epsilon\}=0.
\end{aligned}
\end{equation}
Differentiating~\eqref{eq:higher-k} and using
\eqref{eq:higher-tprime}, we obtain
\[
 k'=-\frac{3n-2(n-1)t}{2(n-1)t}w'
 -\frac{3n(nt-n-t)(2nt-3n+4t)}{4(n-1)^2t^3}w^2.
\]
Substituting this expression and~\eqref{eq:higher-k}
into~\eqref{eq:higher-Gauss} gives
\begin{equation}\label{eq:higher-Gauss-w}
\begin{aligned}
&\frac{3n-2(n-1)t}{2(n-1)t}w'\\
&+\left\{
\frac{3n(nt-n-t)(2nt-3n+4t)}{4(n-1)^2t^3}
-\frac{\{3n-2(n-1)t\}^2}{4(n-1)^2t^2}
\right\}w^2\\
&=\epsilon+na^2t(1-t).
\end{aligned}
\end{equation}
Dividing~\eqref{eq:higher-normal} by $a$ and using
$a'/a=w$, $a''/a=w'+w^2$,
$\lambda=a\{n-(n-1)t\}$, and $\mu=at$, we obtain
\begin{align*}
&-w'-w^2-(n-1)kw+(n-1)k^2\\
&\quad+a^2\left[\{n-(n-1)t\}^2+(n-1)t^2\right]
 -(n+3)\epsilon=0.
\end{align*}
Substituting~\eqref{eq:higher-k} implies
\begin{equation}\label{eq:higher-normal-w}
\begin{aligned}
w'={}&-\frac{6n^2t-9n^2+4nt^2-6nt-4t^2}
 {4t^2(n-1)}w^2\\
&+na^2(nt^2-2nt+n-t^2+2t)-(n+3)\epsilon.
\end{aligned}
\end{equation}
After substituting~\eqref{eq:higher-normal-w}
into~\eqref{eq:higher-Gauss-w} and collecting coefficients,
we have
\begin{equation}\label{eq:higher-linear-one}
\begin{aligned}
&
\frac{3n(8n^2t^2-22n^2t+15n^2+4nt-8t^2)}
 {8t^3(n-1)^2}w^2\\
&-\frac{n(2n^2t^3-7n^2t^2+8n^2t-3n^2-6nt^3
+13nt^2-8nt+4t^3-6t^2)}
 {2t(n-1)}a^2\\
&+\frac{2n^2t-3n^2+2nt-9n-4t}
 {2t(n-1)}\epsilon=0.
\end{aligned}
\end{equation}
Let $C_n(t)$, $D_n(t)$, and $E_n(t)$ be the coefficients of
$w^2$, $a^2$, and $\epsilon$, respectively,
in~\eqref{eq:higher-linear-one}.
Differentiating~\eqref{eq:higher-linear-one}, dividing by $w\ne0$,
and using~\eqref{eq:higher-tprime} and~\eqref{eq:higher-normal-w} yields
\begin{equation}\label{eq:higher-linear-two}
\begin{aligned}
&\left\{\frac{t'}w\frac{dC_n}{dt}
-\frac{6n^2t-9n^2+4nt^2-6nt-4t^2}
 {2t^2(n-1)}C_n\right\}w^2\\
&+\left\{\frac{t'}w\frac{dD_n}{dt}+2D_n
+2n(nt^2-2nt+n-t^2+2t)C_n\right\}a^2\\
&+\left\{\frac{t'}w\frac{dE_n}{dt}
-2(n+3)C_n\right\}\epsilon=0,
\end{aligned}
\end{equation}
where $t'/w$ is given by~\eqref{eq:higher-tprime}. Let
$\widehat C_n(t)$, $\widehat D_n(t)$, and $\widehat E_n(t)$ denote
the coefficients of $w^2$, $a^2$, and $\epsilon$ in~\eqref{eq:higher-linear-two}. Eliminating $w^2$, we find
\begin{equation}\label{eq:higher-eliminate-w}
 \{C_n\widehat D_n-\widehat C_nD_n\}a^2
 +\{C_n\widehat E_n-\widehat C_nE_n\}\epsilon=0.
\end{equation}
From a direct calculation of the coefficients in
\eqref{eq:higher-eliminate-w}, we have
\begin{equation}\label{eq:higher-elimination}
\begin{aligned}
 C_n\widehat D_n-\widehat C_nD_n
 &=\frac{3n^2}{16t^5(n-1)^7}B_n(t),\\
 C_n\widehat E_n-\widehat C_nE_n
 &=\frac{3n^2}{16t^5(n-1)^7}A_n(t).
\end{aligned}
\end{equation}
Here the polynomials $A_n(t)$ and $B_n(t)$ are defined by
\begin{align*}
A_n(t)={}&-16(n-1)^6(n+1)(n+2)t^4\\
&+8(n-1)^5(n+1)(10n^2-17n-68)t^3\\
&-4n(n-1)^4(37n^3-165n^2-438n-82)t^2\\
&+6n^2(n-1)^3(5n+7)(4n^2-39n+17)t\\
&-18n^3(n-1)^3(2n^2-23n-27),
\end{align*}
and
\begin{align*}
B_n(t)={}&16(n-2)(n-1)^6(n+1)(3n+2)t^6\\
&-8n(n-1)^5(n+1)(45n^2-119n+38)t^5\\
&+4n(n-1)^4(280n^4-769n^3+117n^2+380n-116)t^4\\
&-2n^2(n-1)^4(925n^3-2448n^2+525n+422)t^3\\
&+n^3(n-1)^4(1711n^2-4031n+466)t^2\\
&-3n^4(n-1)^3(280n^2-809n+421)t\\
&+9n^5(n-1)^3(19n-43).
\end{align*}
Since $n\ge2$ and $t\ne0$,
\eqref{eq:higher-eliminate-w} and~\eqref{eq:higher-elimination} imply
\begin{equation}\label{eq:higher-AB}
 \epsilon A_n(t)+a^2B_n(t)=0.
\end{equation}

\noindent\textbf{Case 1\textnormal{:} $\epsilon=0$.}
Equation~\eqref{eq:higher-AB} reduces to $B_n(t)=0$.
For $n\ge3$, the leading coefficient of $B_n$ is
\[
 16(n-2)(n-1)^6(n+1)(3n+2),
\]
which is nonzero. For $n=2$, the polynomial $B_2$ has
degree five and leading coefficient $960$. Thus $B_n$ is a nonzero
polynomial for every $n\ge2$, and $t$ is constant.

\medskip
\noindent\textbf{Case 2\textnormal{:} $\epsilon\ne0$.}
Differentiation of~\eqref{eq:higher-AB} gives
\begin{equation}\label{eq:higher-AB-derivative}
 \epsilon A_n'(t)t'
 +a^2\{2wB_n(t)+B_n'(t)t'\}=0.
\end{equation}
Substituting~\eqref{eq:higher-tprime}
into~\eqref{eq:higher-AB-derivative} and
then dividing by $w\ne0$, we obtain
\begin{equation}\label{eq:higher-AB-substituted}
\begin{aligned}
&-\epsilon A_n'(t)
 \frac{(nt-n-t)(2nt-3n+4t)}{2(n-1)t}\\
&+a^2\left\{2B_n(t)
-B_n'(t)\frac{(nt-n-t)(2nt-3n+4t)}{2(n-1)t}\right\}=0.
\end{aligned}
\end{equation}
Eliminating $a^2$ between~\eqref{eq:higher-AB-substituted} and
\eqref{eq:higher-AB}, and using $\epsilon\ne0$ yields
\begin{align*}
Q_n(t):={}&A_n(t)\left\{2B_n(t)
-\frac{B_n'(t)(nt-n-t)(2nt-3n+4t)}{2(n-1)t}\right\}\\
&+\frac{A_n'(t)(nt-n-t)(2nt-3n+4t)}
 {2(n-1)t}B_n(t)=0.
\end{align*}
For $n\ge3$, the polynomial $tQ_n(t)$ has leading coefficient
\begin{equation*}
 512(n-2)(n-1)^{12}(n+1)^3(n+2)(3n+2),
\end{equation*}
which is nonzero. Hence $tQ_n(t)$ is a nonzero polynomial in these
dimensions. For $n=2$, by a direct calculation, we find
\[
 \deg(tQ_2)=10,
 \qquad \text{and its leading coefficient is }368640.
\]
Consequently, $tQ_n(t)$ is a nonzero polynomial for every $n\ge2$,
and $t$ is constant.

Combining the two cases, we conclude that $t$ is constant for every
value of $\epsilon$.

Since $w\ne0$, equation~\eqref{eq:higher-tprime} gives
\begin{equation}\label{eq:higher-candidates}
 t=\frac{n}{n-1}\qquad\text{or}\qquad
 t=\frac{3n}{2(n+2)}.
\end{equation}
For the two values in~\eqref{eq:higher-candidates}, by a direct computation, we obtain
\begin{align*}
 B_n\left(\frac{n}{n-1}\right)&=-2n^5(n-5)(n-3),\\
 B_n\left(\frac{3n}{2(n+2)}\right)
 &=-\frac{4374n^5(n-4)^2(n-1)^3}{(n+2)^6}.
\end{align*}
If $n\ne3,4,5$, equation~\eqref{eq:higher-AB} implies that
$a^2$ is constant. This contradicts $a'\ne0$. The remaining pairs for
which $B_n(t)=0$ are
\begin{equation}\label{eq:higher-exceptional-pairs}
 (n,t)=(3,3/2),\qquad (4,1),\qquad (5,5/4).
\end{equation}
For the other candidate in each of these dimensions, $B_n(t)\ne0$.
Therefore, equation~\eqref{eq:higher-AB} yields a constant value for $a^2$, contrary to
$a'\ne0$. It remains to show that the three exceptional pairs
in~\eqref{eq:higher-exceptional-pairs} are also impossible.

\medskip
\noindent\textbf{Exceptional pair $(n,t)=(3,3/2)$}\textup{:}
We have $k=-w/2$ and $\lambda=0$.
Substitution into~\eqref{eq:higher-Gauss-w} and
\eqref{eq:higher-normal-w}, respectively, gives
\begin{equation}\label{eq:higher-exception-n3}
 \begin{aligned}
 w'&=\frac12w^2+2\epsilon-\frac92a^2,\\
 w'&=\frac12w^2+\frac92a^2-6\epsilon.
 \end{aligned}
\end{equation}
Subtracting the two equations in~\eqref{eq:higher-exception-n3}, we find
$8\epsilon=9a^2$. Thus $a^2$ is constant,
contrary to $a'\ne0$.

\medskip
\noindent\textbf{Exceptional pair $(n,t)=(4,1)$}\textup{:}
Since $A_4(1)=69984\ne0$,
equation~\eqref{eq:higher-AB} forces $\epsilon=0$. Substituting
$n=4$, $t=1$, and $\epsilon=0$
into~\eqref{eq:higher-Gauss-w} and~\eqref{eq:higher-normal-w}, we obtain
\begin{equation}\label{eq:higher-exception-n4}
 \begin{aligned}
 w'&=w^2,\\
 w'&=5w^2+4a^2.
 \end{aligned}
\end{equation}
Subtracting the two equations in~\eqref{eq:higher-exception-n4}
implies $4w^2+4a^2=0$, which is impossible since $a>0$.

\medskip
\noindent\textbf{Exceptional pair $(n,t)=(5,5/4)$}\textup{:}
Since $A_5(5/4)=64000\ne0$,
equation~\eqref{eq:higher-AB} again forces $\epsilon=0$.
Substituting $n=5$, $t=5/4$, and $\epsilon=0$
into~\eqref{eq:higher-Gauss-w} and~\eqref{eq:higher-normal-w} gives
\begin{equation}\label{eq:higher-exception-n5}
 \begin{aligned}
 w'&=\frac12w^2-\frac{25}{8}a^2,\\
 w'&=2w^2+\frac{25}{4}a^2.
 \end{aligned}
\end{equation}
Subtracting the two equations in~\eqref{eq:higher-exception-n5} yields
$3w^2/2+75a^2/8=0$, which is again impossible since $a>0$.

Thus none of the three exceptional pairs occurs. Therefore,
$\grad a=0$ on $\{H\ne0\}$, and hence $|H|$ is constant on $M^n$.
\end{proof}

\section{Proof of the main theorem}\label{sec:classification}

\begin{proof}[Proof of Theorem~\ref{thm:classification-intro}]
By Proposition~\ref{prop:cmc}, $|H|$ is constant.
For $n=2$, Theorem~11 of~\cite{BiharmonicSurface2007} shows that
$\epsilon=1$ and that $f$ is locally congruent to the
model~\eqref{eq:model-intro}, where
\begin{equation}\label{eq:nu-surface}
 \nu^2=\frac{7\pm\sqrt{41}}4.
\end{equation}
The two values in~\eqref{eq:nu-surface} agree with those given
by~\eqref{eq:nu-general} when $n=2$.
For $n\ge3$, Theorem~7 of~\cite{BiminimalHumbilical2012} shows that
$\epsilon=1$, that $f$ is locally congruent to the
model~\eqref{eq:model-intro}, and that $\nu$ satisfies
\eqref{eq:nu-general}. The converse follows from Theorem~11
of~\cite{BiharmonicSurface2007} when $n=2$ and from Remark~9
of~\cite{BiminimalHumbilical2012} when $n\ge3$.
\end{proof}

\begin{corollary}\label{cor:nonpositive}
Let
$f\colon M^n\to\widetilde M^n(4\epsilon)$ be a biharmonic Lagrangian
$H$-umbilical immersion. If $\epsilon\le0$, then
$f$ is minimal.
\end{corollary}

\begin{proof}
By Proposition~\ref{prop:cmc}, $|H|$ is constant. Taking the inner
product of~\eqref{eq:normal-biharmonic} with $H$ and applying the
Weitzenb\"ock formula, by the same computation as in
\cite[proof of Proposition~2.4]{FetcuLoubeauMontaldoOniciuc2010},
we obtain
\[
 |\nabla^\perp H|^2
 +\sum_{i=1}^n|A_He_i|^2-(n+3)\epsilon|H|^2=0.
\]
Since $\epsilon\le0$, this equality implies $A_H=0$. Hence
$n|H|^2=\tr A_H=0$.
\end{proof}

\begin{remark}
For $\epsilon=0$, Corollary~\ref{cor:nonpositive} shows that every
biharmonic Lagrangian $H$-umbilical submanifold of
$\Cn=\mathbb R^{2n}$ is minimal. Thus it provides a partial affirmative
answer to Chen's conjecture.
\end{remark}

\end{document}